\documentclass[12pt, reqno]{amsart}
\usepackage{amsmath, amsthm, amscd, amsfonts, amssymb, graphicx, color}
\usepackage[bookmarksnumbered, colorlinks, plainpages]{hyperref}
\usepackage[all]{xy}
\newtheorem{theorem}{Theorem}[section]
\newtheorem{lemma}[theorem]{Lemma}
\newtheorem{proposition}[theorem]{Proposition}
\newtheorem{corollary}[theorem]{Corollary}
\theoremstyle{definition}
\newtheorem{definition}[theorem]{Definition}
\newtheorem{example}[theorem]{Example}

\theoremstyle{remark}
\newtheorem{remark}[theorem]{Remark}
\numberwithin{equation}{section}

\begin{document}
\setcounter{page}{1}

\title[Generalized injectivity of Banach modules]{Generalized injectivity of Banach modules}

\author[M. Fozouni]{Mohammad Fozouni}

\address{Department of Mathematics,  Gonbad Kavous University, P.O. Box 163, Gonbad-e Kavous, Golestan, Iran.}

\email{\textcolor[rgb]{0.00,0.00,0.84}{fozouni@gonbad.ac.ir}}




\begin{abstract}
In this paper,  we study the notion of $\phi$-injectivity in the special case that $\phi=0$. For an arbitrary locally compact group $G$,  we characterize the 0-injectivity of $L^{1}(G)$ as a left $L^{1}(G)$ module. Also, we show that $L^{1}(G)^{**}$ and $L^{p}(G)$ for $1<p<\infty$  are 0-injective Banach $L^{1}(G)$ modules.
\end{abstract} \maketitle
\noindent
\section{introduction}
The homological properties of Banach modules such as injectivity, projectivity, and flatness was first introduced and investigated by Helemskii; see \cite{Helemskii2, Helemskii}. White in \cite{white} gave a quantitative version of these concepts, i.e., he introduced the concepts of $C$-injective, $C$-projective, and $C$-flat Banach modules for a positive real number $C$. Recently Nasr-Isfahani and Soltani Renani introduce a version of these homological concepts based on a character of a Banach algebra $A$ and they showed that every injective (projective, flat) Banach module is a character injective (character projective, character flat respectively) module but the converse is not valid in general. With use of these new homological concepts, 
they gave a new characterization of $\phi$-amenability of  Banach algebra $A$ such that $\phi\in \Delta(A)$ and a necessary condition for $\phi$-contractibility of $A$; see \cite{Nasr}.

\section{preliminaries}
Let $A$ be a Banach algebra and $\Delta(A)$ denote the character space of $A$, i.e., the space of all non-zero homomorphisms from $A$ onto $\mathbb{C}$. We denote by \textbf{A-mod} and \textbf{mod-A}  the category of all Banach left $A$-modules and all Banach right $A$-modules respectively.  In the case that $A$ has an identity we denote by $\textbf{A-unmod}$ the category of all Banach left unital modules.   For $E, F\in \textbf{A-mod}$, let $_{A}B(E,F)$ be the space of all bounded linear left $A$-module morphisms from $E$ into $F$.

For each Banach space $E$, $B(A,E)$; the Banach algebra consisting of all bounded linear operator from $A$ into $E$, is in $\textbf{A-mod}$ with the following module action:
\begin{equation*}
(a\cdot T)(b)=T(ba)\quad (T\in B(A,E), a, b\in A).
\end{equation*}

\begin{definition}\label{dfn: inj mod} Let $A$ be a Banach algebra and $J\in \textbf{A-mod}$. We say that $J$ is injective if  for each $F, E\in \textbf{A-mod}$ and  admissible monomorphism $T:F\rightarrow E$   the induced map $T_{J}:\ _{A}B(E,J)\rightarrow\ _{A}B(F,J)$ defined by $T_{J}(R)=R\circ T$ is onto.
\end{definition}
Suppose that $\phi\in \Delta(A)$. For $E\in \textbf{A-mod}$, put
\begin{align*}
\ I(\phi,E)&=\textrm{span}\{a\cdot \xi-\phi(\xi)a : a\in A, \xi\in E\},\\
\ _{\phi}B(A^{\sharp},E)&=\{T\in B(A^{\sharp},E) : T(ab-\phi(b)a)=a\cdot T(b-\phi(b)e^{\sharp}), \quad (a, b\in A)\}.
 \end{align*}
It is clear that $I(\phi,E)=\{0\}$ if and only if the module action of $E$ is given by $a\cdot x=\phi(a)x$ for all $a\in A$ and $x\in E$.

Obviously, $\ _{\phi}B(A^{\sharp},E)$ is a Banach subspace of $B(A^{\sharp},E)$. On the other hand, for each $b\in \ker(\phi)$, if $T\in\ _{\phi}B(A^{\sharp},E) $, then $T(ab)=a\cdot T(b)$ for all $a\in A$. Therefore, we conclude that $\ _{\phi}B(A^{\sharp},E)$ is a Banach left $A$-submodule of $B(A^{\sharp},E)$.

Note that if  $E, F\in \textbf{A-mod}$ and $\rho:E\rightarrow F$ is a left $A$-module homomorphism, we can extend the module actions of $E$ and $F$  from $A$ into $A^{\sharp}$ and $\rho$ to a left $A^{\sharp}$-module homomorphism in the following way:
\begin{align*}
&(a,\lambda)\cdot e=a\cdot e+\lambda e\hspace{0.5cm}(a\in A, \lambda \in \mathbb{C}, e\in E)\\
&(a,\lambda)\cdot f=a\cdot f+\lambda f\hspace{0.5cm}(a\in A, \lambda \in \mathbb{C}, f\in F).
\end{align*}
So, $\rho((a, \lambda)\cdot e)=a\cdot \rho(e)+\lambda \rho(e)=(a, \lambda)\cdot \rho(e)$.

For  Banach spaces $E$ and $F$, $T\in B(E,F)$  is admissible if and only if there exists $S\in
B(F,E)$ such that $T\circ S\circ T=T$.

The following definition of a \emph{$\phi$-injective} Banach module,  introduced  by Nasr-Isfahani and Soltani Renani in \cite{Nasr}.
\begin{definition}\label{dfn: char inj mod} Let $A$ be a Banach algebra, $\phi\in \Delta(A)$ and $J\in \textbf{A-mod}$. We say that $J$ is $\phi$-injective if  for each $F, E\in \textbf{A-mod}$ and  admissible monomorphism $T:F\rightarrow E$  with $I(\phi,E)\subseteq \textrm{Im}T$, the induced map $T_{J}$ is onto.
\end{definition}
By Definition \ref{dfn: inj mod} and \ref{dfn: char inj mod}, one can easily check that each injective module is  $\phi$-injective,  although  by \cite[Example 2.5]{Nasr}, the converse is not valid. In \cite{EFL}, the authors with use of the semigroup algebras, gave two good examples  of $\phi$-injective  Banach modules which they are not injective.

Let $E, F$ be in $\textbf{A-mod}$. An operator $T\in\ _{A}B(E,F)$ is called a \emph{retraction} if there exists an $S\in\ _{A}B(F,E)$ such that $T\circ S=Id_{F}$. In this case $F$ is called a retract of $E$. Also, an operator $T\in\ _{A}B(E,F)$ is called a \emph{coretraction} if there exists an $S\in\ _{A}B(F,E)$ such that  $S\circ T=Id_{E}$.

 For $E\in \textbf{A-mod}$, let $\ _{\phi}\Pi^{\sharp}:E\rightarrow\ _{\phi}B(A^{\sharp},E)$ be defined by $_{\phi}\Pi^{\sharp}(x)(a)=a\cdot x$ for all $a\in A^{\sharp}$ and $x\in E$. 

\begin{theorem}\cite[Theorem 2.4]{Nasr}\label{Th: Nasr} Let $A$ be a Banach algebra and $\phi\in \Delta(A)$. For $J\in \emph{\textbf{A-mod}}$ the following statements are equivalent.
\begin{enumerate}

  \item $J$ is $\phi$-injective.
  \item $\ _{\phi}\Pi^{\sharp}\in \ _{A}B(J,\ _{\phi}B(A^{\sharp},J))$ is a coretraction.

\end{enumerate}
\end{theorem}
\section{0-injectivity of Banach modules}
In this section, we give the definition of  a \emph{0-injective} Banach left $A$-module and show that this class of Banach modules are strictly larger than the class of injective Banach modules.

For each $E\in \textbf{A-mod}$ define
  \begin{equation*}
  _{0}B(A^{\sharp},E)=\{T\in B(A^{\sharp},E) : T(ab)=a\cdot T(b)  \text{ for all }a, b\in A \}.
  \end{equation*}
  Clearly, $\ _{0}B(A^{\sharp},E)$ is a Banach left $A$-submodule of $B(A^{\sharp},E)$. It is well-known that $E^{*}$ is in $\textbf{mod-A}$ with the following module action:
  \begin{equation*}
  (f\cdot a)(x)=f(a\cdot x)\hspace{0.5cm}(a\in A, x\in E, f\in E^{*}).
  \end{equation*}
 \begin{definition} Let $A$ be a Banach algebra and $E\in \textbf{A-mod}$. We say that $E$ is (left) 0-injective if  for each $F, K\in \textbf{A-mod}$ and admissible monomorphism $T: F\rightarrow K$ for which $A\cdot K=\textrm{span}\{a\cdot k : a\in A, k\in K\}\subseteq \textrm{Im}T$, the induced map $T_{J}$ is onto.
 
 Similarly, one can define the concept of (right) 0-injective  $A$-module. We say that $E\in \textbf{A-mod}$ is 0-flat if $E^{*}\in \textbf{mod-A}$ is (right) 0-injective.
\end{definition}
Clearly, each injective module is 0-injective.

We use of the following characterization of 0-injectivity  in the sequel without giving the reference.
\begin{proposition} Let $A$ be a Banach algebra and $E\in\emph{\textbf{A-mod}}$. Then $E$ is 0-injective if and only if $\ _{0}\Pi^{\sharp}$ is a coretraction.
\end{proposition}
\begin{proof} Suppose $E\in \textbf{A-mod}$ is 0-injective. Take $F=E$, $K=\ _{0}B(A^{\sharp},E)$ and
$T=\ _{0}\Pi$. Then $A\cdot
K\subseteq\textrm{Im}(_{0}\Pi)$ and $a\cdot
T=\ _{0}\Pi(T(a))$  for each $a\in A$ and $T\in K$.
Hence, for the identity map $I_{E}\in
_{A}\hspace{-0.1cm}B(F,E)=_{A}\hspace{-0.1cm}B(E,E)$, there exists
$\rho\in_{A}\hspace{-0.1cm}B(K, E)=_{A}\hspace{-0.1cm}B(
_{0}B(A^{\sharp},E),E)$ such that $\rho\circ\: _{0}\Pi=\rho\circ
T=I_{E}$.

Conversely, let $_{0}\rho:\ _{0}B(A^{\sharp},E)\longrightarrow E$ be  a left $A$-module morphism and a left inverse for the canonical
morphism $_{0}\Pi$. Suppose that $F, K\in
\textbf{A-mod}$ and $T:F\longrightarrow K$ is an admissible
monomorphism such that  $A\cdot K\subseteq\mathrm{Im}T$. Let
$W\in_{A}\hspace{-0.1cm}B(F,E)$ and define the map
$R:K\longrightarrow\ _{0}B(A^{\sharp},E)$ by $$ R(k)(a)=W\circ
T^{\prime}(a\cdot k) \quad (k\in K, a\in A^{\sharp}),$$ where
$T^{\prime}\in B(K,F)$ satisfies $T\circ T^{\prime}\circ T=T$. We
show that $R$ is well defined, i.e., $R(k)\in\ _{0}B(A^{\sharp},E)$
for each $k\in K$. So, we will show that $R(k)(ab)=a\cdot R(k)(b)$
for each $a, b\in A$. By assumption $A\cdot K\subseteq
\mathrm{Im}T$ and so there exist $f\in F$ such that
$b\cdot k=T(f)$. Therefore
\begin{align*}
a\cdot R(k)(b)&=a\cdot W\circ T^{\prime}(b\cdot k)=a\cdot W\circ
T^{\prime}(T(f))\\
&=a\cdot W(f)=W(a\cdot f)\\
&=W\circ
T^{\prime}(T(a\cdot f))=W\circ T^{\prime}(ab\cdot k)\\
&=R(k)(ab).
\end{align*}
Moreover, for each
$b\in A^{\sharp}$ we have
\begin{align*}
R(a\cdot k)(b)&=W\circ T^{\prime}(b\cdot(a\cdot k))=W\circ T^{\prime}(ba\cdot k)\\
&=R(k)(ba)=(a\cdot R(K))(b).
\end{align*}
It follows that $R(a\cdot k)=a\cdot R(k)$. Now, take $S=\
_{0}\rho\circ R\in\ _{A}B(K, E)$. Since $R\circ T=\ _{0}\Pi\circ W$,
we conclude that $S\circ T=W$, which completes the proof.
\end{proof}
Now,  we  give a sufficient condition for 0-inectivity which provide for us  a large class of Banach algebras $A$ such that they are 0-injective in $\textbf{A-mod}$.

Recall that by \cite[Corollary 2.2.8(i)]{Ramsden}, if $A\in \textbf{A-mod}$ is injective, then $A$ has a right identity. Moreover, the converse is not valid in general even in the case that $A$ has an identity; see Example \ref{Ex}.

\begin{proposition}\label{Prop: sufiicient for 0-inj} Let $A$ be a Banach algebra. If $A$ has an identity, then $A\in \emph{\textbf{A-mod}}$ is 0-injective.
\end{proposition}
\begin{proof} Let $e$ be the identity of $A$. Define $\rho:$$_{0}B(A^{\sharp},A)\rightarrow A$ by $\rho(T)=T(e)$ for all $T\in\ _{0}B(A^{\sharp},A)$. It is obvious that $\rho$ is a left inverse for $_{0}$$\Pi^{\sharp}$, because for each $a\in A$, we have
\begin{align*}
\rho\circ\ _{0}\Pi^{\sharp}(a)=(_{0}\Pi^{\sharp}(a))(e)=ea=a.
\end{align*}
Also, $\rho$ is a left $A$-module morphism, because for each $a\in A$ and $T\in\  _{0}B(A^{\sharp},A)$ we have
\begin{align*}
&\rho(a\cdot T)=(a\cdot T)(e)=T(ea)=T(a)\\
&a\cdot \rho(T)=a\cdot T(e)=T(ae)=T(a).
\end{align*}
Therefore, $A\in \textbf{A-mod}$ is 0-injective.
\end{proof}
For each locally compact group $G$, let $M(G)$ be the Banach algebra consisting of all complex regular Borel measure of $G$ and let $L^{\infty}(G)$ be the space of all measurable complex-valued functions on $G$ which they are essentially bounded; see \cite{Dales} for more details.

The group $G$ is said to be \emph{amenable} if there exists an $m\in L^{\infty}(G)^{*}$ such that $m\geq 0$, $m(1)=1$ and $m(L_{x}f)=m(f)$ for each $x\in G$ and $f\in L^{\infty}(G)$, where $L_{x}f(y)=f(x^{-1}y)$.

As an application of the above theorem we give the following example which shows the difference between 0-injectivity and injectivity.
\begin{example}\label{Ex} Let $G$ be a non-amenable locally compact group. Then by \cite[Theorem 3.1.2]{Ramsden}, $M(G)\in \textbf{M(G)-mod}$ is not injective, but it  is 0-injective. 
\end{example}
By \cite[Proposition VII.1.35]{Helemskii}, if $E\in \textbf{A-unmod}$, each retract of $E$ is injective. For 0-injective Banach modules we have the following proposition.
\begin{proposition}\label{retract} Let $A$ be a Banach algebra and let $E\in \emph{\textbf{A-mod}}$ be 0-injective. Then each retract of $E$ is also 0-injective.
\end{proposition}
\begin{proof} Let $F\in \textbf{A-mod}$ be a retract of $E$. Also, let $T\in\ _{A}B(E,F)$ and $S\in\ _{A}B(F,E)$ be such that $T\circ S=I_{F}$.

Since $E\in \textbf{A-mod}$ is 0-injective, there exists $_{E}\rho^{\sharp}\in\ _{A}B(_{0}B(A^{\sharp},E),E)$ for which $_{E}\rho^{\sharp}\circ\ _{E}\Pi^{\sharp}(x)=x$ for all $x\in E$.

Now, define the map $_{F}\rho^{\sharp}:\ _{0}B(A^{\sharp}, F)\rightarrow F$ by
\begin{equation*}
\ _{F}\rho^{\sharp}(W)=T\circ \ _{E}\rho^{\sharp}(S\circ W)\quad(W\in\ _{0}B(A^{\sharp}, F)).
\end{equation*}
It is straightforward to check that $_{F}\rho^{\sharp}$ is a left $A$-module morphism. On the other hand, for each $y\in F$ we have
\begin{align*}
_{F}\rho^{\sharp}\circ\ _{F}\Pi^{\sharp}(y)&=\ _{F}\rho^{\sharp}(\ _{F}\Pi^{\sharp}(y))\\
&=T\circ \ _{E}\rho^{\sharp}(S\circ \ _{F}\Pi^{\sharp}(y))\\
&=T\circ \ _{E}\rho^{\sharp}(\ _{E}\Pi^{\sharp}(S(y)))\\
&=T\circ S(y)=y.
\end{align*}
Therefore, $F\in \textbf{A-mod}$ is 0-injective.
\end{proof}
Now, we try to characterize the 0-injectivity of $L^{1}(G)$ in $\mathbf{L^{1}(G)}$\textbf{-mod}. First we give the following lemma. 
\begin{lemma}\label{Nec for 0-inj} Let $A$ be a Banach algebra and $E\in \emph{\textbf{A-mod}}$. If $E$ is 0-injective, then
\begin{equation*}
\ _{0}B(A^{\sharp},E)=\{T :  T=R_{x} \text{ on } A \text{ for some } x\in E\},
\end{equation*}
where $R_{x}a=a\cdot x$ for all $a\in A$.
\end{lemma}
\begin{proof}
Let $E\in \textbf{A-mod}$ be 0-injective. So, there exists$\ _{0}\rho^{\sharp}\in\ _{A}B(_{0}B(A^{\sharp},E),E)$ with $_{0}\rho^{\sharp}\circ\ _{0}\Pi^{\sharp}(x)=x$ for all $x\in E$.

Let $T$ be an element of $\ _{0}B(A^{\sharp},E)$. Hence
\begin{align*}
T(b)=\ _{0}\rho^{\sharp}\circ\ _{0}\Pi^{\sharp}(T(b))&=\ _{0}\rho^{\sharp}(\ _{0}\Pi^{\sharp}(T(b)))\\
&=\ _{0}\rho^{\sharp}(b\cdot T)\\
&=b\cdot \ _{0}\rho^{\sharp}(T).
\end{align*}
Take $x_{0}=\ _{0}\rho^{\sharp}(T)$. So, $T=R_{x_{0}}$ on $A$ and this completes the proof.
\end{proof}
Recall that $E\in \textbf{A-mod}$ is faithful in $A$, if for each $x\in E$,  the relation $a\cdot x=0$ for all $a\in A$, implies $x=0$.
\begin{theorem}\label{0-inj of L1} Let $G$ be a locally compact group. Then $L^{1}(G)\in \mathbf{L^{1}(G)}$\emph{\textbf{-mod}} is 0-injective if and only if $G$ is discrete.
\end{theorem}
\begin{proof} Let $G$ be a discrete group. Then $L^{1}(G)$ is unital and so the result follows from Proposition \ref{Prop: sufiicient for 0-inj}.

Conversely, let $G$ be  non-discrete. So, $L^{1}(G)\neq M(G)$. Suppose that $\mu\in M(G)\setminus L^{1}(G)$. Since $L^{1}(G)$ is an ideal of $M(G)$, the operator $T_{\mu}$ defined by
\begin{equation*}
T_{\mu}((f, \lambda))=f\cdot \mu\quad((f,\lambda)\in L^{1}(G)^{\sharp}),
\end{equation*}
 is in $\ _{0}B(L^{1}(G)^{\sharp}, L^{1}(G))$, but it is not of the form $R_{x}$ for some $x\in L^{1}(G)$, because $M(G)$ is  faithful in $L^{1}(G)$. Therefore, by Lemma \ref{Nec for 0-inj}, $L^{1}(G)$ in $\mathbf{L^{1}(G)}$\textbf{-mod} is not 0-injective.
\end{proof}

Recall that a Banach algebra $A$ is  left 0-amenable if for every Banach $A$-bimodule $X$ with $a\cdot x=0$ for all $a\in A$ and $x\in X$, every continuous derivation $D:A\rightarrow X^{*}$ is inner,  or equivalently, $H^{1}(A, X^{*})=0$ where  $H^{1}(A, X^{*})$ denotes the first cohomology group of $A$ with coefficients in $X^{*}$; see \cite{Nasr-Studia} for more details.

Now, we investigate the relation between 0-injectivity and 0-amenability.

Let $E, F\in \textbf{A-mod}$. Suppose that  $Z^{1}(A\times E, F)$ denotes the Banach space of all continuous bilinear maps $B: A\times E\longrightarrow F$ satisfying 
\begin{equation*}
a\cdot B(b, \xi)-B(ab, \xi)+B(a, b\cdot \xi)=0\quad (a, b\in A, \xi\in E).
\end{equation*}
 Define $\delta_{0}: B(E, F)\longrightarrow Z^{1}(A\times E, F)$ by $(\delta_{0}T)(a, \xi)=a\cdot T(\xi)-T(a\cdot \xi)$ for all $a\in A$ and $\xi \in E$. Then we have
 \begin{equation*}
 \mathrm{Ext}_{A}^{1}(E, F)=Z^{1}(A\times E, F)/\textrm{Im}\delta_{0}.
 \end{equation*}
By \cite[Proposition VII.3.19]{Helemskii2}, we know that $\mathrm{Ext}_{A}^{1}(E, F)$ is topologically isomorphic to $H^{1}(A, B(E, F))$ where $B(E, F)$ is a Banach $A$-bimodule with the following module actions:
\begin{equation*}
(a\cdot T)(\xi)=a\cdot T(\xi),\quad (T\cdot a)(\xi)=T(a\cdot \xi)\hspace{0.5cm}(a\in A, \xi\in E, T\in B(E,F)).
\end{equation*}
 
 To see further details about $\mathrm{Ext}_{A}^{1}(E, F)$; see \cite{Helemskii}.

\begin{lemma}\label{Lemma}
Let  $E\in \emph{\textbf{A-mod}}$. If $\mathrm{Ext}_{A}^{1}(F,E)=\{0\}$ for all $F\in \emph{\textbf{A-mod}}$ with $A\cdot F=0$, then $E\in \emph{\textbf{A-mod}}$ is 0-injective.
\end{lemma}
\begin{proof}
 To show this,  let $K, W\in \textbf{A-mod}$ and $T:K\rightarrow W$ be an admissible monomorphism with $A\cdot W\subseteq \textrm{Im}T$. We claim that the induced map $T_{E}$  is onto.

We know that the short complex  $0\rightarrow K\xrightarrow{T} W\xrightarrow{q} \frac{W}{\textrm{Im}T}\rightarrow 0$ is admissible where $q$ is the quotient map. But for all $a\in A$ and $x\in W$, $a\cdot (x+\textrm{Im}T)=\textrm{Im}T$, because $A\cdot W\subseteq \textrm{Im}T$. Therefore, by assumption $\textrm{Ext}_{A}^{1}(\frac{W}{\textrm{Im}T},E)=\{0\}$. Now, by \cite[III Theorem 4.4]{Helemskii}, the  complex 
\begin{equation*}
0\rightarrow\ _{A}B(\frac{W}{\textrm{Im}T},E)\rightarrow\ _{A}B(W,E)\xrightarrow{T_{E}}\ _{A}B(K,E)\rightarrow \textrm{Ext}_{A}^{1}(\frac{W}{\textrm{Im}T},E)\rightarrow \cdots,
\end{equation*}
is exact. Therefore, $T_{E}$ is onto.

\end{proof}
Recall that if $E, F$ be two Banach spaces and $E\widehat{\otimes}F$ denotes the projective tensor product space, then $(E\widehat{\otimes}F)^{*}$ is isomorphic to $B(E, F^{*})$ as two Banach spaces with the pairing
\begin{equation*}
<Tx, y>=T(x\otimes y)\hspace{0.5cm}(x\in E, y\in F, T\in (E\widehat{\otimes}F)^{*}).
\end{equation*}
Also, note that $E\widehat{\otimes}F$ is  isometrically isomorphic to $F\widehat{\otimes}E$ as two Banach spaces. 
\begin{theorem}\label{Th: 0-amenable} Let $A$ be a Banach algebra. Then $A$ is left 0-amenable if and only if each $J\in \emph{\textbf{mod-A}}$ is 0-flat.
\end{theorem}
\begin{proof} 
Suppose that $A$ is left 0-amenable. We show that $\textrm{Ext}_{A}^{1}(E,J^{*})=\{0\}$ for all $E\in\textbf{A-mod}$ with $A\cdot E=0$. We have
\begin{equation*}
\textrm{Ext}_{A}^{1}(E, J^{*})=H^{1}(A,B(E,J^{*}))=H^{1}(A, (E\widehat{\otimes}J)^{*})=\{0\},
\end{equation*}
because $E\widehat{\otimes}J\in \textbf{mod-A}$ has the module action, $a\cdot z=0$ for all $z\in E\widehat{\otimes}J$. Therefore,  by Lemma \ref{Lemma}, $J^{*}\in \textbf{A-mod}$ is 0-injective.

Conversely, let $J\in \textbf{mod-A}$ be 0-flat. So, for Banach right $A$-module $\mathbb{C}$ with module action $\lambda\cdot a=0$ for all $a\in A$ and  $\lambda\in \mathbb{C}$ we have
\begin{align*}
H^{1}(A, J^{*})=H^{1}(A, B(J, \mathbb{C}))&=H^{1}(A, B(J, \mathbb{C}^{*}))\\
&=H^{1}(A, (J\widehat{\otimes}\mathbb{C})^{*})\\
&=H^{1}(A, (\mathbb{C}\widehat{\otimes}J)^{*})\\
&=H^{1}(A, B(C,J^{*}))\\
&=\textrm{Ext}_{A}^{1}(\mathbb{C},J^{*})\\
&=0.
\end{align*}
Hence, if we take $J$ a left $A$ module with module action $a\cdot x=0$ for all $a\in A$ and $x\in J$, then the above relation implies that $A$ is 0-amenable.
\end{proof}
By \cite[Corollary 4.7]{Dales. Polyakov}, we know that $L^{1}(G)^{**}\in \mathbf{L^{1}(G)}$\textbf{-mod} is injective if and only if $G$ is an amenable group. Also, if $1<p<\infty$ by \cite[Theorem 9.6]{DDPR}, $L^{p}(G)\in \mathbf{L^{1}(G)}$\textbf{-mod} is injective if and only if $G$ is an amenable group.
\begin{corollary} Let $G$ be a locally compact group, $1<p<\infty$ and $E\in \mathbf{L^{1}(G)}$\emph{\textbf{-mod}} be  $L^{p}(G)$ or $L^{1}(G)^{**}$. Then $E\in \mathbf{L^{1}(G)}$\emph{\textbf{-mod}} is 0-injective.
\end{corollary}
\begin{proof}
Since $L^{1}(G)$ has a bounded approximate identity by \cite[Proposition 3.4 (i)]{Nasr-Studia}, we know that $L^{1}(G)$ is 0-amenable. So, by Theorem \ref{Th: 0-amenable} we conclude the result. The second part follows similarly, because for each $1<p<\infty$ we know that $L^{q}(G)^{*}=L^{p}(G)$ where $q$ satisfies the relation $q^{-1}+p^{-1}=1$.
\end{proof}
\begin{remark}
In general, by \cite[Proposition 3.4 (i)]{Nasr-Studia}, if $A$ is a Banach algebra with a bounded approximate identity, then each $E\in \textbf{mod-A}$ is 0-flat.
\end{remark}
\bibliographystyle{amsplain}

\begin{thebibliography}{99}
\bibitem{Dales}  H.G. Dales, \textit{Banach Algebras and Automatic Continuity}, Clarendon Press, Oxford, 2000.

\bibitem{Dales. Polyakov}  H.G. Dales, M.E. Polyakov, \textit{Homological properties of modules over group algebras}, Proc. London Math. Soc. 89 (2004), 390--426.

\bibitem{DDPR} H.G. Dales, M. Daws, H.L. Pham,  P. Ramsden, \textit{Multi-norms and the injectivity of $L^{p}(G)$}, J. London Math. Soc. (2012) 86 (3), 779--809.

\bibitem{EFL}  M. Essmaili, M. Fozouni, J. Laali, \textit{Hereditary properties of character injectivity with application to semigroup algebras}. Ann. Funct. Anal. 6 (2015), no 2, 162--172.


\bibitem{Helemskii2}  A.Ya. Helemskii, \textit{Banach and Locally Convex Algebras}, Clarendon Press, Oxford, 1993.

\bibitem{Helemskii} ----------, \textit{The Homology of Banach and Topological Algebras}, Kluwer Academic
Publishers Group, Dordrecht, 1989.

\bibitem{Nasr-Studia} R. Nasr-Isfahani, S. Soltani Renani, \textit{Character contractibility of Banach algebras and
homological properties of Banach modules}, Studia Math. 202 (2011), 205--225.

\bibitem{Nasr} ----------, \textit{Character injectivity and projectivity of Banach modules}, Quart. J. Math. (2014) 65 (2), 665--676.

\bibitem{Ramsden Paper} P. Ramsden, \textit{Homological properties of modules over semigroup algebras}, J. Funct. Anal. 258 (2010) 3988--4009.

\bibitem{Ramsden} ----------, \textit{Homological properties of semigroup algebras}, Ph. D. Thesis, University of Leeds, 2008.

\bibitem{white} M.C. White, \textit{Injective modules for uniform algebras}, Proc. London Math. Soc. (3) 73 (1996), 155--184.
\end{thebibliography}

\end{document}